\documentclass[12pt]{amsart}

\usepackage{amssymb,amsfonts}
\usepackage{amsthm}
\usepackage{graphicx}
\usepackage[all,arc]{xy}
\usepackage{hyperref}
\hypersetup{colorlinks,allcolors=magenta}
\usepackage{enumerate}
\usepackage{bbm}
\usepackage{dsfont}
\usepackage{mathrsfs}
\usepackage{tikz-cd}
\usetikzlibrary{shapes}
\usepackage{import}
\usepackage{float}
\restylefloat{table}
\usepackage{multirow}
\usepackage{pdflscape}

\usepackage[left=1in,top=1in,right=1in,bottom=1in]{geometry}

\usepackage{mathtools}
\newtheorem{thm}{Theorem}[section]
\newtheorem*{thm*}{Theorem}

\newtheorem{innercustomthm}{Theorem}
\newenvironment{customthm}[1]
  {\renewcommand\theinnercustomthm{#1}\innercustomthm}
  {\endinnercustomthm}
  
\newtheorem{innercustomcor}{Corollary}
\newenvironment{customcor}[1]
  {\renewcommand\theinnercustomcor{#1}\innercustomcor}
  {\endinnercustomcor}
  
\newtheorem{cor}[thm]{Corollary}
\newtheorem{prop}[thm]{Proposition}
\newtheorem{lem}[thm]{Lemma}

\newtheorem{quest}[thm]{Question}

\theoremstyle{definition}
\newtheorem{defn}[thm]{Definition}

\newtheorem*{thm1.2}{\textrm{Theorem 1.2}}

\theoremstyle{remark}
\newcommand{\Mbar}{\overline{\mathcal{M}}}
\newcommand{\M}{\mathcal{M}}

\newtheorem{sch}[thm]{Scholium}

\newcommand{\Z}{\mathbb{Z}}
\newcommand{\Q}{\mathbb{Q}}

\newcommand{\bbS}{\mathbb{S}}

\newcommand{\ch}{\operatorname{ch}}

\def\A{\mathbb{A}}

\def\C{\mathbb{C}}

\def\Q{\mathbb{Q}}

\def\Z{\mathbb{Z}}

\def\calA{\mathcal{A}}

\def\calM{\mathcal{M}}

\def\calV{\mathcal{V}}
\def\calW{\mathcal{W}}
\def\calX{\mathcal{X}}
\def\calY{\mathcal{Y}}
\def\calZ{\mathcal{Z}}

\newcommand{\Ind}{\operatorname{Ind}}
\newcommand{\Res}{\operatorname{Res}}
\newcommand{\Spec}{\operatorname{Spec}}
\newcommand{\Exp}{\operatorname{Exp}}
\newcommand{\Log}{\operatorname{Log}}

\makeatletter
\let\c@equation\c@thm
\makeatother
\numberwithin{equation}{section}

\graphicspath{ {images/} }
\title{Equivariant Hodge polynomials of heavy/light moduli spaces}

\author[S. Kannan]{Siddarth Kannan}\address{Department of Mathematics, Brown University, Providence, RI 02906}
\email{\url{siddarth_kannan@brown.edu}}
\author[S. Serpente]{Stefano Serpente}\address{Dipartimento di Matematica e Fisica, Università Roma Tre, Rome, I-00146}\email{\url{stefano.serpente@uniroma3.it}}
\author[C. Yun]{Claudia He Yun}\address{MPI MiS, 04103 Leipzig, Germany}
\email{\url{clyun@mis.mpg.de}}

\begin{document}
\begin{abstract}
Let $\Mbar_{g, m|n}$ denote Hassett's moduli space of weighted pointed stable curves of genus $g$ for the \textit{heavy/light} weight data $\left(1^{(m)}, 1/n^{(n)}\right)$, and let $\M_{g, m|n} \subset \Mbar_{g, m|n}$ be the locus parameterizing smooth, not necessarily distinctly marked curves. We give a change-of-variables formula which computes the generating function for $(S_m\times S_n)$-equivariant Hodge--Deligne polynomials of these spaces in terms of the generating functions for $S_{n}$-equivariant Hodge--Deligne polynomials of $\Mbar_{g,n}$ and $\M_{g,n}$. 
\end{abstract}
	
\maketitle\thispagestyle{empty}
%\tableofcontents

\section{Introduction}
Given nonnegative integers $g$, $m$, and $n$ satisfying
\[2g - 2 + m + \mathrm{min}(n,1) > 0, \]
set $\Mbar_{g,m|n}$ to be Hassett's moduli space $\Mbar_{g, \mathcal{A}}$ of weighted pointed stable curves of genus $g$, for the weight data 
\[\mathcal{A} =\left(\underbrace{1,\ldots, 1}_{m}, \underbrace{1/n, \ldots, 1/n}_{n} \right).\]
This space is a connected, smooth, and proper Deligne-Mumford stack over $\Z$, and is a compactification of the moduli space $\mathcal{M}_{g, m+n}$ of smooth pointed algebraic curves of genus $g$ \cite{Hassett}; this family of weight data has been called \textit{heavy/light} in the literature \cite{CHMRtropical, KKLChow}. We also set $\calM_{g, m|n} \subset \Mbar_{g, m|n}$ to be the locus of smooth, not necessarily distinctly marked curves. In this paper we study the $(S_m \times S_n)$-equivariant Hodge--Deligne polynomials of $\calM_{g, m|n}$ and $\Mbar_{g, m|n}$. Throughout this paper we will work with the coarse moduli spaces of these stacks, as the mixed Hodge structure on the rational cohomology of a Deligne-Mumford stack coincides with that of its coarse moduli space.

If $X$ is a $d$-dimensional complex variety with an action of $S_m \times S_n$, its complex cohomology groups are $(S_m \times S_n)$-representations in the category of mixed Hodge structures. The $(S_m \times S_n)$-equivariant Hodge--Deligne polynomial of $X$ is given by the formula
\begin{equation}\label{HodgeDelignePolynomial}
E_X^{S_m \times S_n}(u,v) := \sum_{i,p,q = 0}^{2d} (-1)^{i} \ch_{m,n}\left(\mathrm{Gr}^F_p\mathrm{Gr}^{W}_{p + q} H^{i}_c(X; \C)\right) u^p v^q \in \Lambda^{(2)}[u,v],  \end{equation}
where $\Lambda^{(2)} = \Lambda \otimes \Lambda$ is the ring of bisymmetric functions, $\ch_{m, n}(V) \in \Lambda^{(2)}$ is the Frobenius characteristic of an $(S_m \times S_n)$-representation $V$, while $W$ and $F$ denote the weight and Hodge filtrations on the compactly supported cohomology of a complex algebraic variety, respectively. The Hodge--Deligne polynomial has also been referred to as the $E$-polynomial and the mixed Hodge polynomial in the literature. If $X$ is proper, and the mixed Hodge structure on each cohomology group is pure, as is the case for the coarse moduli space of $\Mbar_{g, m|n}$, the Hodge--Deligne polynomial specializes to the usual Hodge polynomial
\[\sum_{p,q = 0}^{2d} (-1)^{p+ q}\ch_{m, n}\left(H^{p,q}(X;\C)\right) u^p v^q. \]
For more details on mixed Hodge structures, see \cite{PetersSteenbrink} or \cite{MotivicIntegration}.

We assemble all of the equivariant Hodge--Deligne polynomials for heavy/light Hassett spaces with fixed genus into series with coefficients in $\Lambda^{(2)}$:
\[\mathsf{a}_g := \sum_{m,n} E^{S_m \times S_n}_{\calM_{g, m|n}}(u,v),\quad\overline{\mathsf{a}}_g := \sum_{m,n} E^{S_m \times S_n}_{\Mbar_{g, m|n}}(u,v) \in \Lambda^{(2)}[[u,v]]. \]
We also define
\[
\mathsf{b}_g := \sum_{n} E_{\calM_{g, n}}^{S_n}(u,v), \quad \overline{\mathsf{b}}_g := \sum_{n} E_{\Mbar_{g, n}}^{S_n}(u,v) \in \Lambda[[u,v]].
\]
In the above, for a variety $X$ with action of $S_n$, we have set $E_{X}^{S_n}(u,v)$ for the $S_n$-equivariant Hodge--Deligne polynomial of $X$, defined analogously to (\ref{HodgeDelignePolynomial}), replacing $\ch_{m,n}$ with the Frobenius characteristic $\ch_n$ of an $S_n$-representation, and replacing $\Lambda^{(2)}$ with $\Lambda$.

In order to state our main theorem on the above generating functions, we require some combinatorial preliminaries. Given a symmetric function $f \in \Lambda$, we set $f^{(j)} \in \Lambda^{(2)}$ for the inclusion of $f$ into the $j$th tensor factor, $j\in\{1,2\}$. These extend to maps $\Lambda[[u,v]] \to \Lambda^{(2)}[[u, v]]$. Let $p_i \in \Lambda$ be the $i$th power sum symmetric function. The coproduct $\Lambda \to \Lambda^{(2)}$ defined by
\[p_i \mapsto p_i^{(1)} + p_i^{(2)} \]
also extends to a map $\Delta: \Lambda[[u,v]] \to \Lambda^{(2)}[[u,v]]$. There are two plethysm operations $\circ_1, \circ_2$ defined on $\Lambda^{(2)}$, and these extend to $\Lambda^{(2)}[[u, v]]$ by 
\[p_n^{(i)} \circ_i q = q^n, \]
\[p_n^{(i)} \circ_j q = p_n^{(i)}, \]
for $\{i,j\} = \{1,2\}$ and $q \in \{u, v\}$. See Section \ref{SymmetricBackground} for more details and references on symmetric functions and the Frobenius characteristic.

The main contributions of this paper are the following formulas, which encode the combinatorial relationships between the generating functions defined above.
\begin{customthm}{A}\label{GeneratingFunctionFormulas}
Let $h_n \in \Lambda$ denote the $n$th homogeneous symmetric function. For $f \in \Lambda[[u, v]]$ set
\[\Exp(f) = \sum_{n > 0} h_n \circ f. \]
Then we have
\[\mathsf{a}_g = \Delta(\mathsf{b}_g) \circ_2 \Exp\left(p_1\right)^{(2)} \] and
\[\overline{\mathsf{a}}_g = \Delta(\overline{\mathsf{b}}_g) \circ_2 \left(p_1 - \frac{\partial \mathsf{b}_0}{\partial p_1} \right)^{(2)} \circ_2 \Exp\left(p_1\right)^{(2)}. \]
\end{customthm}
A formula for the series $\mathsf{b}_0$ has been given by Getzler \cite{GetzlerGenusZero}; therefore Theorem \ref{GeneratingFunctionFormulas} determines $\mathsf{a}_g$ and $\overline{\mathsf{a}}_g$ in terms of $\mathsf{b}_g$ and $\overline{\mathsf{b}}_g$. Moreover, this transformation is invertible, as $\Exp$ has a plethystic inverse $\Log$ and $p_1 - \partial \mathsf{b}_0/\partial p_1$ is inverse to $p_1 + \partial \overline{\mathsf{b}}_0/\partial p_1$. There is a numerical analogue of Theorem \ref{GeneratingFunctionFormulas} which deals with the non-equivariant Hodge--Deligne polynomials, defined by the assignment
\[E_X(u, v) :=  \sum_{i,p,q = 0}^{2d} (-1)^{i} \dim\left(\mathrm{Gr}^F_p\mathrm{Gr}^{W}_{p + q} H^{i}_c(X; \C)\right) u^p v^q \in \Q[u,v].  \]
Set

\[a_g :=  \sum_{m,n} E_{\mathcal{M}_{g, m|n}}(u, v) \frac{x^m y^n}{m! n!}, \quad \overline{a}_g := \sum_{m,n} E_{\Mbar_{g, m|n}}(u, v) \frac{x^m y^n}{m! n!} \in \Q[[u, v, x, y]], \]
and similarly put
\[b_g := \sum_{n} E_{\mathcal{M}_{g, n}}(u, v) \frac{x^n}{n!}, \quad  \overline{b}_g := \sum_{n} E_{\Mbar_{g, n}}(u, v) \frac{x^n}{n!} \in \Q[[u, v, x]]. \]

\begin{customcor}{B}\label{NumericalFormula}
We have
\[a_g = b_g|_{x \to w} \]
where $w =x + e^y - 1$,
and
\[\overline{a}_g = \overline{b}_g|_{x \to z}, \]
where
\[z = x + e^{y} + \frac{e^{uvy} - uv\cdot e^{y} + uv - 1}{uv - u^2v^2} - 1. \]
\end{customcor}

Theorem \ref{GeneratingFunctionFormulas} and Corollary \ref{NumericalFormula} allow for many explicit computations. For example, Getzler ~\cite{GetzlerSemiClassical} gives a recursive formula for the calculation of $\overline{\mathsf{b}}_1$, which allows us to compute $\overline{\mathsf{a}}_1$ and $\overline{a}_1$; sample calculations are included at the end of the paper in Tables \ref{EquivariantGenusOneTable} and \ref{NumericalGenusOneTable}. Similarly, Chan et al. in \cite{CFGP} give a formula for the $S_n$-equivariant weight 0 compactly supported Euler characteristic of $\calM_{g, n}$ in arbitrary genus, so Theorem \ref{GeneratingFunctionFormulas} gives a practical method to compute the $(S_m \times S_n)$-equivariant weight 0 compactly supported Euler characteristic of $\calM_{g, m|n}$. Sample computations for $g = 2$ have been included in Table \ref{WeightZeroGenusTwoTable}.

\subsection{Context}
The heavy/light moduli space $\Mbar_{g, m|n}$ has been studied in several algebro-geometric contexts. It is of interest in its own right, as a modular compactification of $\calM_{g, m+n}$ which admits a birational morphism from the Deligne--Mumford--Knudsen moduli space $\Mbar_{g, m+n}$. It may be viewed as a resolution of singularities of the $n$-fold product of the universal curve over $\Mbar_{g, m}$ \cite{Janda}. It arises in the theory of stable quotients ~\cite{MOP} and in tropical geometry ~\cite{CHMRtropical, MUW, HahnLi}. As $g$, $m$, and $n$ vary, the spaces $\Mbar_{g, m|n}$ form the components of Losev--Manin's extended modular operad \cite{LosevManinOperad}; when $g = 0$ and $m = 2$, the space $\Mbar_{0, 2|n}$ is a toric variety, and it coincides with the Losev--Manin moduli space of stable chains of $\mathbb{P}^1$'s \cite{LosevManinToric}.

Part of the motivation for our work is to generate new data on the symmetric group representations afforded by the cohomology  of these moduli spaces, and to understand how changing the weight data affects the cohomological complexity. For example, our Corollary \ref{cor: chiComparison} shows that asymptotically, the all-light-points compactification $\Mbar_{1, 0|n}$ has significantly less cohomology than the Deligne--Mumford compactification $\Mbar_{1, n}$. We ask if the same phenomenon holds in higher genus (Question \ref{asymptoticsQuestion}). In hopes that our data will help others prove new theorems about the cohomology of these moduli spaces, we also ask (Question \ref{weightzeroQuestion}) whether one can find a closed formula for the equivariant weight zero compactly supported Euler characteristics of the open moduli spaces $\M_{g, m|n}$, as is done for $\M_{g, n}$ in \cite{CFGP}. Our theorem in principle allows one to calculate these Euler characteristics for arbitrary $g$, $m$, and $n$, but as our formulas involve plethysm, this is a very difficult computational task.

\subsection{Related work}
The first part of Corollary \ref{NumericalFormula} follows from an observation made in \cite{tropicalhassett}: that in the Grothendieck ring of varieties, one has an equality
\[[\calM_{g, m|n} ] = \sum_{k = 1}^{n} S(n, k) [\calM_{g, m + k}],  \]
where $S(n, k)$, the Stirling number of the second kind, counts the number of partitions of $\{1, \ldots, n\}$ with $k$ parts. It follows that the generating function $a_g$ can be obtained from $b_g(x + t)$ by making the substitution $t = e^{y} - 1$; this transformation is called the \textit{Stirling transform}. Both parts of Theorem \ref{GeneratingFunctionFormulas} involve  plethysm with the symmetric function $\Exp(p_1)^{(2)}$, which transforms to $e^{y} - 1$ under the \textit{rank} homomorphism $\Lambda \otimes \Lambda \to \Q[[x, y]]$. Thus Theorem \ref{GeneratingFunctionFormulas} can be viewed as an application of the equivariant version of the Stirling transform.

In genus zero, the problem of computing the equivariant Hodge polynomials of $\Mbar_{0, m|n}$ has been studied by Bergstr{\"o}m--Minabe \cite{BergstromMinabe1, BergstromMinabe2} and by Chaudhuri \cite{Chaudhuri}. Our formula gives a third approach to this problem, which applies in arbitrary genus. 

Also in genus zero, the Chow groups of $\Mbar_{0, m|n}$ have been computed by Ceyhan \cite{Ceyhan}. The Chow ring has been computed by Petersen \cite{PetersenChow} and by Kannan--Karp--Li \cite{KKLChow}. Our combinatorial proofs of plethystic formulas are similar to those in Petersen's work \cite[\S 4]{PetersenTautological} on moduli spaces of genus two curves of compact type, and their tautological rings.

The techniques of this paper are based on prior work on the operad structure of moduli of stable curves and maps, by Getzler \cite{GetzlerGenusZero, GetzlerSemiClassical}, Getzler--Kapranov \cite{GetzlerKapranov}, and Getzler--Pandharipande \cite{GetzlerPandharipande}. In particular, the main tools of the paper are Getzler--Pandharipande's theory of $\bbS$-spaces, which encode sequences of varieties with $S_n$-actions, together with a careful analysis of permutation group actions on the boundary strata of the spaces $\Mbar_{g, m|n}$. To carry this out, we use the language of $\bbS^2$-spaces, which are a mild generalization of Getzler--Pandharipande's theory. In this way our approach is similar to that of Chaudhuri \cite{Chaudhuri}, who uses the language of \textit{$\bbS^2$-modules}. To understand the cohomology of the open moduli space with its mixed Hodge structure, Chaudhuri uses the Leray spectral sequence associated to the fibration $\M_{0, m|n} \to \M_{0, m|n-1}$. However, this technique seems to be limited to the genus zero case. We restrict ourselves to the study of the Hodge--Deligne polynomial of $\M_{g, m|n}$, which contains less information than the full cohomology ring with its mixed Hodge structure.

\subsection{Outline of the paper} We review the necessary background on symmetric functions, the Frobenius characteristic, and Hassett spaces in Section \ref{background}. In Section \ref{GrothendieckRing}, we define the Grothendieck ring of $\bbS^2$-spaces and its composition operations as categorifications of plethysm. We then use these composition operations to prove Theorem \ref{GeneratingFunctionFormulas} in Section \ref{MainProof}. Explicit calculations and accompanying remarks are included at the end of the paper, in Section \ref{calculations}.

 \subsection*{Acknowledgements} We thank Madeline Brandt, Melody Chan, and Dhruv Ranganathan for useful conversations. SK was supported by an NSF Graduate Research Fellowship.

\section{Background}\label{background}
Here we briefly recall some background on symmetric functions and on Hassett spaces.

\subsection{Symmetric functions and the Frobenius characteristic}\label{SymmetricBackground} For a more detailed background on the ring of symmetric functions, see Macdonald \cite{Macdonald}, Stanley \cite{Stanley}, or Getzler--Kapranov \cite[\S 7]{GetzlerKapranov}.
The ring $\Lambda$ of symmetric functions over $\Q$ is defined as \[\Lambda := \lim_{\longleftarrow} \Q[[x_1, \ldots, x_n]]^{S_n}.\] Elements of $\Lambda$ are power series which are invariant under any permutation of the variables. We have that
\[\Lambda = \Q[[p_1, p_2, \ldots]] \]
where $p_i = \sum_{k > 0} x_k^i$ is the $i$th power sum symmetric function. The ring $\Lambda$ is graded by degree, and $p_i$ has degree $i$. One can view $\Lambda$ as the Grothendieck ring of \textit{$\bbS$-modules}, where an $\bbS$-module $\calV$ is the data of an $S_n$-representation $\calV(n)$ for each $n \geq 0$. The ring structure on the Grothendieck ring of $\bbS$-modules is induced by the tensor product:
\[ (\calV \otimes \calW)(n) = \bigoplus_{j = 0}^{n} \Ind_{S_j \times S_{n - j}}^{S_n} \calV(j) \otimes \calW(n-j), \]
where $\Ind$ denotes induction of representations. Given an $\bbS$-module $\calV$, the Frobenius characteristic $\ch(\calV)$ is defined by
\[\ch (\calV) := \sum_{n \geq 0} \ch_n(\calV(n))\],
where for an $S_n$-representation $V$, we have
\[ \ch_n(V) := \frac{1}{n!} \sum_{\sigma \in S_n} \mathrm{Tr}_V(\sigma) p_{\lambda(\sigma)}, \]
where $\lambda(\sigma)$ is the cycle type of the permutation $\sigma$, and for a partition $\lambda \vdash n$ we set $p_\lambda := \prod_{i} p_{\lambda_i}$. The Frobenius characteristic induces a ring isomorphism
\[ \ch : K_0(\bbS\text{-modules}) \to \Lambda,  \]
where $K_0(\bbS\text{-modules})$ is the aforementioned Grothendieck ring of $\bbS$-modules. In particular, if $V$ is an $S_n$-representation, then $\ch_n(V)$ determines $V$: the Schur functions $s_{\lambda}$ for $\lambda \vdash n$ form a basis for the homogeneous degree $n$ part of $\Lambda$, and if
\[V = \bigoplus_{\lambda \vdash n} W_{\lambda}^{\oplus a_\lambda} \]
is the decomposition of $V$ into Specht modules, then
\[\ch_n(V) = \sum_{\lambda \vdash n} a_{\lambda} s_{\lambda}. \]
We define the homogeneous symmetric functions $h_n$ by
\[ h_n := \ch_{n}(\mathrm{Triv}_n), \]
where $\mathrm{Triv}_n$ is the trivial $S_n$-representation of dimension one. Note that $h_n = s_n$.

There is an associative operation $\circ$, called \textit{plethysm}, on $\Lambda$. The plethysm $f \circ g$ is defined when $f$ has bounded degree, or when the degree $0$ term of $g$ vanishes. It is characterized by the following formulas:
\begin{itemize}
    \item[(i)] $(f_1 + f_2) \circ g = f_1 \circ g + f_2 \circ g$,
    \item[(ii)] $(f_1 f_2) \circ g = (f_1 \circ g)(f_2 \circ g)$
    \item[(iii)] if $f = f(p_1, p_2, \ldots)$, then $p_n \circ f = f(p_n, p_{2n}, \ldots)$.
\end{itemize}

Plethysm has an interpretation on the level of Frobenius characteristics: given two $\bbS$-modules $\calV$ and $\calW$ with $\calW(0) = 0$, define a third $\bbS$-module $\calV \circ \calW$ by the formula
\[(\calV \circ \calW)(n) := \bigoplus_{j \geq 0} \calV(j) \otimes_{S_j} \calW^{\otimes j}(n). \]
Then $\ch(\calV \circ \calW) = \ch(\calV) \circ \ch(\calW)$; see \cite[Proposition 7.3]{GetzlerKapranov} or \cite[Chapter 7, Appendix 2]{Stanley}.

All of these constructions generalize to $(S_m \times S_n)$-representations. As in the introduction, we set \[\Lambda^{(2)}:= \Lambda \otimes \Lambda;\] we call $\Lambda^{(2)}$ the ring of \textit{bisymmetric functions}. Given $f \in \Lambda$, we write $f^{(j)}$ for the inclusion of $f$ into the $j$th tensor factor. Then we have
\[\Lambda^{(2)} = \Q[[p_1^{(1)}, p_1^{(2)}, p_2^{(1)}, p_2^{(2)}, \ldots ]]. \]
Define an $\bbS^2$-module $\calV$ to be the data of an $(S_m \times S_n)$-representation $\calV(m, n)$ for each $m, n \geq 0$. Given an $(S_{m} \times S_n)$-representation $V$, its Frobenius characteristic is the bisymmetric function
\label{frobcharmn}\[\ch_{m,n}(V) := \frac{1}{m!\cdot n!} \sum_{(\sigma, \tau) \in S_{m} \times S_n} \mathrm{Tr}_V(\sigma, \tau) p^{(1)}_{\lambda(\sigma)} p^{(2)}_{\lambda(\tau)}. \]
Just as in the single variable case, the bisymmetric function $\ch_{m,n}(V)$ completely determines the $(S_m \times S_n)$-representation $V$: if
\[V = \bigoplus_{\substack{{\lambda \vdash m}\\{\mu \vdash n}}} \left(W_{\lambda} \otimes W_{\mu} \right)^{\oplus a_{\lambda \mu}},  \]
then
\[\ch_{m,n}(V) = \sum_{\lambda, \mu} a_{\lambda \mu} s_{\lambda}^{(1)} s_{\mu}^{(2)}. \]
The Frobenius characteristic of a $\bbS^2$-module $\calV$ is defined by
\[ \ch(\calV) := \sum_{m,n \geq 0} \ch_{m, n}(\calV(m, n)). \]
It furnishes a ring isomorphism from the Grothendieck ring of $\bbS^2$-modules to $\Lambda^{(2)}$.
The ring $\Lambda^{(2)}$ has two plethysm operations $\circ_1$ and $\circ_2$; the operation $f \circ_i g$ is defined whenever $f$ has bounded degree, or the degree $(0,0)$ term of $g$ vanishes. These operations are characterized by:
\begin{itemize}
\item[(i)] $(f_1 + f_2) \circ_i g = f_1 \circ_i g + f_2 \circ_i g$,
\item[(ii)] $(f_1 f_2) \circ_i g = (f_1 \circ_i g)(f_2 \circ_i g)$;
\item[(iii)] if $f = f(p_1^{(1)}, p_1^{(2)}, p_2^{(1)}, p_2^{(2)} \ldots)$, then $p_n^{(i)} \circ_i f = f(p_n^{(1)}, p_n^{(2)}, p_{2n}^{(1)}, p_{2n}^{(2)}, \ldots )$, for any $i, j \in \{1,2\}$, and
\item[(iv)] $p_n^{(i)} \circ_j f = p_n^{(i)}$ if $i \neq j$;
\end{itemize}
see Chaudhuri ~\cite{Chaudhuri}. We will make use of the following interpretation of plethysm, for which we are not aware of a suitable reference in the bisymmetric case. Let $\calV$ be an $\bbS^2$-module, and let $\calW$ be an $\bbS$-module with $\calW(0) = 0$. We can compose these in two ways to get an $\bbS^2$-module:
\begin{align*}
    (\calV \circ_1 \calW)(m, n) &= \bigoplus_{j \geq 0} \calV(j, n) \otimes_{S_j} \calW^{\otimes j}(m)\\
     (\calV \circ_2 \calW)(m, n) &= \bigoplus_{j \geq 0} \calV(m, j) \otimes_{S_j} \calW^{\otimes j}(n)
\end{align*}
We will interpret plethysm of bisymmetric functions in terms of these composition operations.
\begin{prop}\label{BisymmetricPlethysm}
    Let $\calV$ be an $\bbS^2$-module, and let $\calW$ be an $\bbS$-module. Then
    \[\ch(\calV \circ_1 \calW) = \ch(\calV) \circ_1 \ch(\calW)^{(1)} \] and
    \[ \ch(\calV \circ_2 \calW) = \ch(\calV) \circ_2 \ch(\calW)^{(2)}. \]
\end{prop}
\begin{proof}
    Since
    \[ (\calV_1 \oplus \calV_2) \circ_i \calW \cong (\calV_1 \circ_i \calW) \oplus (\calV_2 \circ_i \calW), \]
    it suffices to consider the case where $\calV$ is supported in a single degree $(m, k)$ and $\calV(m, k)$ is irreducible, so that $\calV(m, k) \cong V_\lambda \otimes V_\mu$, where $\lambda \vdash n$, $\mu \vdash k$, and $V_\lambda, V_\mu$ are the respective Specht modules. We will also suppose that $i = 2$. In this case we have
    \begin{align*}
    (\calV \circ_2 \calW)(m, n) &=  \calV(m, k) \otimes_{S_k} \calW^{\otimes k}(n) \\&\cong (V_\lambda \otimes V_\mu) \otimes_{S_k} \calW^{\otimes k}(n) \\&\cong V_\lambda \otimes (V_\mu \otimes_{S_k} \calW^{\otimes k}(n)) \\&\cong V_\lambda \otimes (V_\mu \circ \calW)(n).
    \end{align*}
    Taking $\ch$ on both sides, we see that
    \begin{align*}
        \ch(\calV \circ_2 \calW) &= \ch(V_\lambda) \otimes \ch (V_\mu \circ \calW) \\&=  \ch(V_\lambda) \otimes (\ch (V_\mu) \circ \ch(\calW)) \\&= \left(\ch(V_\lambda) \otimes \ch(V_\mu)\right) \circ_2 1 \otimes \ch(\calW) \\&= \ch(\calV) \circ_2 (1 \otimes \ch(\calW));
    \end{align*} 
the case of the operation $\circ_1$ is similar.
\end{proof}
The ring $\Lambda$ is a Hopf algebra, with coproduct $\Delta: \Lambda \to \Lambda^{(2)}$ defined by
\[p_i \mapsto p_i^{(1)} + p_i^{(2)}. \]
On the level of Frobenius characteristic, we have
\begin{equation}\label{CoprodFrobenius}
    \Delta(\ch_n(V)) = \sum_{k = 0}^{n} \ch_{k, n-k}\left(\Res^{S_n}_{S_{k} \times S_{n - k}} V\right),
\end{equation}
where $\Res$ denotes restriction of representations. There is a rank homomorphism
\begin{equation}\label{Rank}
\mathrm{rk}: \Lambda \to \Q[[x]],
\end{equation}
determined by
\[\ch_n(V) \mapsto \dim(V) \cdot \frac{x^n}{n!}, \]
or equivalently, $p_1 \mapsto x$ and $p_n \mapsto 0$ for $n > 1$. This takes plethysm into composition of power series. We use the same notation for the morphism
\begin{equation}\label{BisymmetricRank}
\mathrm{rk}: \Lambda^{(2)} \to \Q[[x, y]]
\end{equation}
determined by 
\[\ch_{m,n}(V) \mapsto \dim(V) \cdot \frac{x^m y^n}{m! n!}, \]
or $p_1^{(1)} \mapsto x$, $p_1^{(2)} \mapsto y$, and $p_n^{(j)} \mapsto 0$ for $n > 1$. In this case, the two plethysm operations $\circ_1$ and $\circ_2$ are carried into composition in $x$ and $y$, respectively.
\subsection{Hassett spaces}
Let $g\geq 0$, $n \geq 1$ be two integers, and let $\mathcal{A} =(a_1 ,\dots, a_n )\in ((0,1]\cap \mathbb{Q})^n$ be a weight datum such that
$2g-2 + a_1 +\dots+a_n >0.$
Let $C$ be a curve, at worst nodal, with $p_1 , \dots , p_n$ smooth points of $C$. We say that $(C,p_1,\dots, p_n)$ is  \textit{$\mathcal{A}$-stable} if
\begin{enumerate}[(i)]
    \item the twisted canonical sheaf $K_C + a_1p_1 +...+a_np_n$ is ample;
    \item whenever a subset of the marked points $p_i$ for $i\in S \subset \{1,\dots, n\}$ coincide, we have $\sum_{i\in S} a_i \leq 1.$
\end{enumerate}

Equivalently, condition (i) is that for each irreducible component $E$ of $C$, we have
\begin{equation}
    2g(E)- 2 + |E\cap \overline{C \smallsetminus E}| + 2|\mathrm{Sing}(E)| + \sum_{i\mid p_i\in E}a_i > 0,
\label{eqn: stability condition}
\end{equation} see \cite[Proposition 3.3]{ulirsch2015tropical}. Hassett shows that there exists a connected Deligne-Mumford stack $\overline{\mathcal{M}}_{g,\mathcal{A}}$ of dimension $3g - 3 + n$, smooth and proper over $\mathbb{Z}$, which parameterizes $\mathcal{A}$-stable curves of genus $g$ \cite{Hassett}. When $\mathcal{A} =(1,\dots, 1)$ is a sequence of $n$ ones, Hassett stability coincides
with the Deligne--Mumford--Knudsen stability, and $\overline{\mathcal{M}}_{g,\mathcal{A}} =\overline{\mathcal{M}}_{g,n}$. Each stack $\Mbar_{g, \mathcal{A}}$ is equipped with a birational morphism
\[\rho_{\calA} : \Mbar_{g, n} \to \Mbar_{g, \calA}. \]
% These Hassett spaces are modular compactifications  of
% the moduli space $\mathcal{M}_{g,n}$ of smooth curves with $n$ marked points.

% We set $\mathcal{M}_{g,\mathcal{A}}\subset \overline{\mathcal{M}}_{g,\mathcal{A}}$ for the locus of smooth curves. 

In this paper we are interested in the family  of weight data
\[\calA_{m,n} = \left(\underbrace{1,\ldots, 1}_{m}, \underbrace{1/n, \ldots, 1/n}_{n} \right),\]
and as in the introduction, we put $\Mbar_{g, m|n}$ for the resulting moduli space, called the heavy/light Hassett space. We say that a curve is \textit{$(m|n)$-stable} if it is $\mathcal{A}_{m,n}$-stable. We now characterize $(m|n)$-stability in combinatorial terms.

\begin{defn}
For $(C, p_1, \ldots, p_{m + n}) \in \Mbar_{g, m + n}$, let $T \subset C$ be a union of irreducible components of $C$. We say $T$ is a \textit{rational tail} if $T$ is a connected curve of arithmetic genus zero, and $T$ meets $\overline{C \smallsetminus T}$ in a single point.
\end{defn}
% \begin{defn}
% Let $S \subseteq \{1, \ldots, m+n \}$. Given $(C, p_1, \ldots, p_{m + n}) \in \Mbar_{g, m + n}$, we say a rational tail $T \subset C$ \textit{supports} $S$ if for each $i \in S$, we have $p_i \in T$.
% %\[S = \{i \mid p_i \in T \}. \]
% \end{defn}
% \noindent Given a rational tail $T$, we say that an irreducible component $E$ of $T$ is a \textit{middle} component if $|E\cap \overline{C\smallsetminus E}| = 2$, and we say it is \textit{terminal} if $|E\cap \overline{C\smallsetminus E}|=1$. The following lemma determines $(m|n)$-stability in terms of rational tails.
The following lemma is a straightforward application of (\ref{eqn: stability condition}), so we omit the proof.
\begin{lem}\label{CombinatorialStability}
Let $(C, p_1, \ldots, p_ {m+n}) \in \Mbar_{g, m+n}$. Then $C$ is $(m|n)$-stable if and only if every rational tail of $C$ contains at least one marking with index in $\{1, \ldots, n\}$.
\end{lem}
% \begin{proof}
% First assume $(C, p_1, \ldots, p_{m+n})$ is $(m|n)$-stable. Then each of its irreducible components $E$ satisfies \[2g(E)- 2 + |(E\cap \overline{C\smallsetminus E})\cup \mathrm{Sing}(E)| + \sum_{i\mid p_i\in E}a_i > 0.\] If $T$ is a rational tail, it must have at least one terminal component. The above inequality reduces to $\sum_{i|p_i\in E}a_i > 1$ for such a component, so the marked points on $T$ cannot be only subset of the last $n$, since the sum of their weights would be at most 1.

% Now assume $(C, p_1, \ldots, p_{m+n})$ is not $(m|n)$-stable.
% %We show there exists a rational tail which supports only a subset of the last $n$ markings.
% %Since $(C, p_1, \ldots, p_ {m+n}) \in \Mbar_{g, m+n}$, it is stable. Each of its irreducible components $E$ satisfies \[2g(E) - 2 + |(E\cap \overline{C\smallsetminus E})\cup \mathrm{Sing}(E)| + |S| > 0,\] where $S$ is the set of indices supported on $E$. 
% Let $E$ be a component of $C$. The inequality \eqref{eqn: stability condition} is satisfied if $g(E)\geq 1$ or $g(E) = 0$ with $|(E\cap \overline{C\smallsetminus E})|\geq 2$.
% %verify the inequality for $(m|n)$-stability, so the only way for $(C, p_1,\dots, p_{m+n})$ to be not $(m|n)$-stable is that 
% Therefore, there must be a rational tail $T$ consisting of a single component such that $\sum_{i|p_i\in T}a_i \leq 1$, and this can happen only if $T$ supports $S$ with $S\subseteq \{m+1,\dots m+n\}$.
% \end{proof}
\section{The Grothendieck ring of $\bbS^2$-spaces}\label{GrothendieckRing}
A $G$-variety is a variety with an action of a group $G$.
%In their work on Betti numbers of Kontsevich spaces of stable maps
An $\bbS$-space $\calX$ is a sequence of $S_n$-varieties $\calX(n)$ for $n \geq 0$.
%, together with an action of $S_n$ on $X_n$ for each $n$. 
Getzler--Pandharipande define a Grothendieck ring of $\mathbb{S}$-spaces \cite{GetzlerPandharipande}. We briefly generalize this formalism to the case of $(S_m \times S_n)$-varieties. First, we define the Grothendieck group
\[ K_0(\mathsf{Var}, S_m \times S_n) \]
of $(S_m \times S_n)$-varieties. %(a $G$-variety is a variety together with an action of the group $G$). 
This group is constructed by first taking the free abelian group generated by isomorphism classes of $(S_m \times S_n)$-varieties, and then imposing the relation
\[[X] = [X \smallsetminus Y] + [Y] \]
whenever $Y$ is an $(S_m \times S_n)$-invariant subvariety of $X$. We define an \textit{$\bbS^2$-space} $\calX$ to be a collection of varieties $\calX(m,n)$ together with an action of $S_m \times S_n$ for each pair $(m,n)$ with $m,n\geq 0$. We refer to $\calX(m,n)$ as the \textit{arity} $(m,n)$ component of $\calX$. We define the Grothendieck group of $\bbS^2$-spaces as the product
\[ K_0(\mathsf{Var}, \bbS^2) := \prod_{m,n \geq 0} K_0(\mathsf{Var}, S_m \times S_n).  \]
We can make $K_0(\mathsf{Var}, \bbS^2)$ into a ring using the $\boxtimes$-product on $\bbS^2$-spaces:
\[(\calX \boxtimes \calY)(m,n) = \coprod_{i = 0}^{m} \coprod_{j = 0}^{n} \Ind_{S_i \times S_{m - i} \times S_{j} \times S_{n - j}}^{S_{m} \times S_n} \calX(i, j) \times \calY(m-i,n-j). \]
This generalizes the box product on $K_0(\mathsf{Var}, \bbS)$, given by
\[ (\calX \boxtimes \calY)(n) = \coprod_{j = 0}^{n} \Ind_{S_j \times S_n-j}^{S_n} \calX(j) \times \calY(n-j). \]

The ring $K_0(\mathsf{Var}, \bbS^2)$ is an algebra over the subring $K_0(\mathsf{Var}) = K_0(\mathsf{Var}, S_0 \times S_0)$, which is nothing but the usual Grothendieck group of varieties. We will make use of a certain composition structure on the Grothendieck ring of $\bbS^2$-spaces: we will compose an $\bbS^2$-space with an $\bbS$-space, as follows. 
Given $\calX$ a $\bbS^2$-space and $\calY$ an $\bbS$-space with $\mathcal{Y}(0) = \varnothing$, we can define two composition operations, $\circ_1$ and $\circ_2$:
\begin{equation}\label{circone}
    (\calX \circ_1 \calY)(m,n) = \coprod_{i = 0}^{\infty} (\calX(i, n) \times \calY^{\boxtimes i}(m))/S_i,
\end{equation}
and
\begin{equation}\label{circtwo}
    (\calX \circ_2 \calY)(m,n) =  \coprod_{j = 0}^{\infty}(\calX(m, j) \times \calY^{\boxtimes j}(n))/S_j.
\end{equation}
Given an $\bbS^2$-space $\mathcal{X}$, we define its \textit{Hodge--Deligne series} by
\[\mathsf{e}(\mathcal{X}) := \sum_{m,n \geq 0} E^{S_m \times S_n}_{\mathcal{X}(m,n)}(u,v) \in \Lambda^{(2)}[[u, v]]. \]
The lift of this series to the Grothendieck ring of mixed Hodge structures has been called the Serre characteristic or the Hodge-Grothendieck characteristic ~\cite{Bagnarol}.

The composition operations (\ref{circone}) and (\ref{circtwo}) should be viewed as categorifications of plethysm in the following sense: if $\calX$ is an $\bbS^2$-space, and $\calY$ is an $\bbS$-space, then
\begin{equation}\label{CompositionWorks}
\mathsf{e}(\mathcal{X} \circ_i \mathcal{Y}) = \mathsf{e}(\mathcal{X}) \circ_i \mathsf{e}(\mathcal{Y})^{(i)} 
\end{equation}
for $i = 1, 2$, where for an $\bbS$-space $\calZ$, one defines
\[\mathsf{e}(\calZ) :=\sum_{n \geq 0} E^{S_n}_{\calZ(n)}(u, v) \in \Lambda[[u,v]]. \]
This follows upon upgrading Proposition \ref{BisymmetricPlethysm} from $\bbS^2$-modules in vector spaces to $\bbS^2$-modules in the category of graded rational mixed Hodge structures, as is done for $\bbS$-modules in Getzler \cite{GetzlerPreprint} and Getzler--Pandharipande \cite[\S5]{GetzlerPandharipande}.
We will put $\nu_n$ for the $\bbS$-space supported in arity $n$, where it is given by $\Spec \C$ with trivial action of $S_n$. Note that
\[\mathsf{e}\left(\nu_n\right) = h_n. \]
Given an $\bbS$-space $\calX$ with $\calX(0) = \varnothing$, define an analogue of the exponential function $e^x - 1$ by
\[\Exp(\mathcal{X}) := \sum_{n > 0} \nu_n \circ \mathcal{X}. \]
Finally, given an $\bbS$-space $\calZ$, we define an $\bbS^2$-space $\Delta \calZ$ by the assignment
\[\Delta\calZ(m, n) := \Res^{S_{m + n}}_{S_m \times S_n} \calZ(m + n). \]
 We have
\begin{equation}\label{CoproductForHodge}
    \mathsf{e}(\Delta \calZ) = \Delta(\mathsf{e}(\calZ)),
\end{equation}
by (\ref{CoprodFrobenius}).

% Next, we put
% \[I_1\calZ(m,n):=  \begin{cases} 
% \calZ(m) &\text{ if } n = 0, \\
% \varnothing &\text{ else}
% \end{cases}\]
% and
% \[I_2\calZ(m,n):= \begin{cases} 
% \calZ(n) &\text{ if } m = 0,\\
% \varnothing &\text{ else}.
% \end{cases},\]
% so
% \[\mathsf{e}(I_j \calZ) = \mathsf{e}(\calZ)^{(j)} \]
% for $j =1,2$. The $\bbS$-space $\nu_n$, which contains $\Spec \C$ with trivial $S_n$-action in arity $n$, and $\varnothing$ elsewhere, satisfies $ I_j(\nu_n) = \nu_n^{(j)}$. 

\section{Proof of Theorem \ref{GeneratingFunctionFormulas}}\label{MainProof}
Our main theorem is proven using the composition operations defined above. First, for each $g \geq 0$, define $\bbS^2$-spaces as follows:

\begin{equation}\label{open-hl-ssquared}
\calM_g^{\mathrm{hl}}(m,n) = \begin{cases} \calM_{g, m|n} &\text{ if } 2g - 2 + m + \text{min}(n, 1) > 0, \\
\varnothing &\text{ else},
\end{cases}
\end{equation}

\begin{equation}\label{closed-hl-ssquared}
\overline{\calM}_g^{\mathrm{hl}}(m,n) = \begin{cases} 
\overline{\calM}_{g, m|n} &\text{ if } 2g - 2 + m + \text{min}(n, 1) > 0, \\
\varnothing &\text{ else}.
\end{cases}
\end{equation}

We will also make use of the $\bbS$-spaces
\[\calM_g(n) = \begin{cases} 
\calM_{g,n} &\text{ if } 2g - 2 + n > 0, \\
\varnothing &\text{ else},
\end{cases} \]
and
\[\overline{\calM}_g(n) = \begin{cases} 
\overline{\calM}_{g,n} &\text{ if } 2g - 2 + n > 0, \\
\varnothing &\text{ else}.
\end{cases} \]

\begin{prop}\label{OpenModuliSpace}
We have
\[[\calM_g^{\mathrm{hl}}] = [\Delta\calM_g \circ_2 \Exp\left(\nu_1 \right)]\]
\end{prop}
\begin{proof}
Define the $\bbS$-space $\calY$ by \[\calY = \Exp\left(\nu_1\right).\]
We have that $\mathcal{Y}(0) = \varnothing$, while $\calY(n) = \nu_n$ for $n \geq 1$. For any $\bbS^2$-space $\calX$, we have
\begin{align*}
    (\calX \circ_2 \calY)(m,n) &= \coprod_{j = 0}^{\infty} \left(\calX(m, j) \times \calY^{\boxtimes j}(n) \right)/S_j \\&= \coprod_{j = 0}^{\infty} \left(\calX(m, j) \times \coprod_{\substack{{k_1 + \cdots + k_j = n}\\{k_r > 0 \,\forall\, r}}} \Ind_{S_{k_1} \times \cdots \times S_{k_j}}^{S_n} \Spec \C \right)/S_j.
\end{align*}

Now let us return to the $S_m \times S_n$ space $\calM_{g, m|n}$. This space admits a stratification: for $1 \leq j \leq n$, let $\calZ_{m,j} \subset \calM_{g, m|n}$ denote the locally closed stratum in which there are precisely $j$ distinct marked points among the last $n$. Then we can write
\begin{equation}\label{openstratum}
    \calZ_{m,j} \cong \left(\coprod_{\substack{{k_1 + \cdots + k_j = n}\\{k_r > 0 \, \forall \, r}}} \Res^{S_{m + j}}_{S_m \times S_j} \calM_{g, m + j} \times \Ind_{S_{k_1} \times \cdots \times S_{k_j}}^{S_n} \Spec \C\right)/S_j.
\end{equation} 
Since
\[[\calM_{g}^{\mathrm{hl}}(m, n)] = \sum_{j = 1}^{n} [\calZ_{m,j}], \]
we see that
\[[\calM_g^{\mathrm{hl}}] = [\Delta\calM_g \circ_2 \Exp\left(\nu_1\right)] \]
upon summing over $j$, $m$, and $n$ on both sides of (\ref{openstratum}).
\end{proof}

Towards proving our theorem for the compact moduli space $\Mbar_{g}^{\mathrm{hl}}$, it is useful to introduce an auxiliary moduli space.
\begin{defn}
We set
\[\Mbar_{g, n}^{(k)} \subset \Mbar_{g, n} \]
to be the locus of curves which have no rational tails whose support consists of any subset of the last $k$ markings.
We define an $\bbS^2$-space $\Mbar^\star_g$ by
\[\Mbar^{\star}_{g}(m, n) : = \Mbar_{g, m + n}^{(n)}. \]
\end{defn}
The following proposition expresses the $\bbS^2$-space $\Delta \Mbar_g$ in terms of $\Mbar^{\star}_g$ and the composition operation. The basic idea has appeared in the literature before, in the main theorem of ~\cite{GetzlerSemiClassical}; see also \cite{semiclassicalremark}. For an $\bbS$-space $\calZ$, we put
\[\delta \calZ(n) := \Res^{S_{n + 1}}_{S_n} \calZ(n+1); \]
note that
\[\mathsf{e}(\delta \calZ) = \frac{\partial \mathsf{e}(\calZ)}{\partial p_1}, \]
by \cite[Proposition 8.10]{GetzlerKapranov}.

\begin{prop}\label{RationalTails}
Let $\Mbar_0^{\dagger}$ denote the $\bbS$-space that is empty in arities $0$ and $1$, supports $\Spec \C$ in arity $2$, and supports $\Mbar_{0, n }$ in arities $n \geq 3$. Then we have
\[\left[\Delta\Mbar_g\right]  = \left[\Mbar^\star_g \circ_2 \delta \Mbar_0^{\dagger}\right].   \]
\end{prop}
\begin{proof}
Let $\calX$ denote the class on the right-hand side of the claimed equality, and note that $\calY= \delta \Mbar_0^{\dagger}$ is the $\bbS$-space which supports $\varnothing$ in arity $0$, $\nu_1$ in arity $1$, and $\mathrm{Res}^{S_{n + 1}}_{S_n} \M_{0, n + 1}$ in arities $n \geq 2$.
A point of the $S_n$-space $\calY^{\boxtimes j}(n)$ corresponds to an ordered tuple of varieties
\[(X_1, \ldots, X_j) \] such that:
\begin{enumerate}
    \item for all $i$, $X_i$ is either $\Spec \C$ or a pointed stable curve of arithmetic genus zero whose marked points are labelled by $\{0, \ldots, r_i\}$ for some $r_i \geq 2$;
    \item there is a chosen bijection: \[\{X_i \mid X_i = \Spec \C \} \cup \{p \mid p \text{ is a nonzero marked point of }X_j \text{ for some }j \} \to \{1, \dots, n\}.\]
\end{enumerate}
The group $S_n$ acts on the chosen bijection, and $S_j$ acts by reordering the tuple. Now recall that

\begin{align*}
\calX(m,n) &= \coprod_{j = 0}^{\infty} \left(\Mbar^\star_g(m, j) \times \calY^{\boxtimes j}(n)\right)/S_j
\\&= \coprod_{j = 0}^{n} \left(\Mbar^\star_g(m, j) \times \calY^{\boxtimes j}(n)\right)/S_j.
\end{align*}
We can see that the class $[\calX(m, n)]$ in the Grothendieck group of $(S_m \times S_n)$-varieties is equal to $[\Res^{S_{m + n}}_{S_m \times S_n} \Mbar_{g, m+n}]$. Indeed, for fixed $j$, one takes the ordered tuple $(X_1, \ldots, X_j)$ represented by a point of $\calY^{\boxtimes j}(n)$, and glues in the indicated order to the $j$ distinguished marked points of a pointed curve in $\Mbar^\star_g(m, j)$. This has the effect of adding rational tails which support subsets of the final $n$ markings. Taking the quotient by the diagonal action of $S_j$ makes this gluing procedure into an isomorphism between $\left(\Mbar^\star_g(m, j) \times \calY^{\boxtimes j}(n)\right)/S_j$ and the stratum of $\Res^{S_{m + n}}_{S_m \times S_n} \Mbar_{g, m + n}$ where there are exactly $j$ rational tails supporting subsets of the last $n$ markings, where we allow for ``trivial" rational tails which are just marked points.
\end{proof}
The final ingredient of the proof of Theorem \ref{GeneratingFunctionFormulas} is the following formula, analogous to Proposition \ref{OpenModuliSpace}.
\begin{prop}\label{NRTtoHL}
We have
\[[\Mbar_g^\star \circ_2 \Exp\left(\nu_1\right)] = [\Mbar_g^{\mathrm{hl}}] \]
\end{prop}
\begin{proof}
The proof is essentially the same as that of Proposition \ref{OpenModuliSpace}: stratify $\Mbar_{g, m|n}$ by
\[\calW_{m, j} = \{(C, p_1, \ldots, p_{m + n}) \mid \text{ there are } j \text{ distinct points among the last }n  \}, \]
and observe that
\[\calW_{m ,j} \cong \coprod_{\substack{{k_1 + \cdots + k_j = n}\\k_r > 0\,\forall\,r }} \left(\Mbar_{g, m + j}^{(j)} \times \Ind_{S_{k_1} \times \cdots \times S_{k_j}}^{S_n} \Spec \C \right)/S_j, \]
by Lemma \ref{CombinatorialStability}. The proof is complete upon summing over $m$ and $j$.
\end{proof}

We can now prove the main theorem.
\begin{proof}[Proof of Theorem \ref{GeneratingFunctionFormulas}]
The first part of the theorem follows from taking $\mathsf{e}(\cdot)$ on both sides of Proposition \ref{OpenModuliSpace} and using both (\ref{CompositionWorks}) and (\ref{CoproductForHodge}). From Proposition \ref{RationalTails} and (\ref{CoproductForHodge}), we see that
\begin{equation}\label{ApplyingHodgeSeries}
\Delta(\overline{\mathsf{a}}_g) = \mathsf{e}(\Mbar^{\star}_g) \circ_2 \left(p_1 + \frac{\partial \overline{\mathsf{b}}_0}{\partial p_1} \right)^{(2)},
\end{equation}
as $\mathsf{e}(\nu_1) = p_1$ and $\mathsf{e}(\delta \Mbar_0) = \partial \overline{\mathsf{b}}_0/\partial p_1$. The symmetric functions
\[p_1 + \frac{\partial \overline{\mathsf{b}}_0}{\partial p_1} \mbox{ and } p_1 - \frac{\partial \mathsf{b}_0}{\partial p_1} \]
are plethystic inverses; this is because $\mathsf{b}_0$ and $\overline{\mathsf{b}}_0$ are \textit{Legendre transforms} of one another, as explained in \cite{GetzlerGenusZero}. We thus perform the operation
\[\circ_2 \left(p_1 - \frac{\partial \mathsf{b}_0}{\partial p_1}\right)^{(2)} \]
on both sides of (\ref{ApplyingHodgeSeries}) to see that
\begin{equation*}
    \Delta(\overline{\mathsf{b}}_g) \circ_2  \left(p_1 - \frac{\partial \mathsf{b}_0}{\partial p_1}\right)^{(2)} = \mathsf{e}(\Mbar^{\star}_g).
\end{equation*}
The theorem is now proven upon applying Proposition \ref{NRTtoHL}. 
\end{proof}
To prove Corollary \ref{NumericalFormula}, one uses the rank morphisms (\ref{Rank}) and (\ref{BisymmetricRank}). We apply $\mathrm{rk}$ to both sides of Theorem \ref{GeneratingFunctionFormulas}, and use that
\[\mathrm{rk}\left(\Exp\left(p_1\right)^{(2)} \right) = e^{y} - 1. \]
The corollary follows from the formula
\begin{align*}
    \mathrm{rk}\left(\left(p_1 - \frac{\partial \mathsf{b}_0}{\partial p_1}\right)^{(2)}\right) &= y - \sum_{n \geq 2} E_{\calM_{0, n + 1}}(u, v) \cdot \frac{y^n}{n!} \\&= y + \frac{(y + 1)^{uv}- uvy - 1}{uv - u^2v^2},
\end{align*}
due to Getzler \cite{GetzlerGenusZero}.

\section{Calculations}\label{calculations}
\subsection{The Euler characteristic of $\Mbar_{1, 0|n}$} We begin this section by proving two results on the topological Euler characteristic $\chi(\Mbar_{1, 0|n})$, which may be viewed as corollaries of Theorem \ref{GeneratingFunctionFormulas} and Getzler's semi-classical approximation \cite{GetzlerSemiClassical}. The space $\Mbar_{1,0|n}$ is interesting to compare with $\Mbar_{1, n}$, as it parameterizes curves that have no rational tails. The first result determines the generating function for the numbers $\chi(\Mbar_{1, 0|n})$.
\begin{prop}
Define
\[f(y) : = \sum_{n \geq 1} \chi(\Mbar_{1, 0|n})\frac{y^n}{n!}. \]
Then
\[f(y) = -\frac{y}{12} - \frac12\log(1-y) + \varepsilon \circ (e^y-1), \] 
where
\[\varepsilon(y) := \frac{1}{12}(19y + 23y^2/2 + 10y^3/3 + y^4/2).  \]
\end{prop}

\begin{proof}
Apply $\mathrm{rk}$ to both sides of Theorem \ref{GeneratingFunctionFormulas} and consider the $x$-degree 0 part. We obtain an equality
\begin{equation*}
    \sum_n E_{\Mbar_{g,0|n}}(u,v) \frac{y^n}{n!} = \left(\sum_n E_{\Mbar_{g,n}}(u,v) \frac{y^n}{n!}\right) \circ \left(y-\sum_{n\geq 2} E_{\M_{0,n+1}}(u,v) \frac{y^n}{n!}\right) \circ (e^y-1).
\end{equation*}
We substitute $u=v=1$ and $g=1$:
\begin{equation*}
    \sum_n \chi(\Mbar_{1,0|n}) \frac{y^n}{n!} = \left(\sum_n \chi(\Mbar_{1,n}) \frac{y^n}{n!}\right) \circ \left(y-\sum_{n\geq 2} \chi(\M_{0,n+1}) \frac{y^n}{n!}\right) \circ (e^y-1).
\end{equation*}
Let
\[g(y) := y + \sum_{n \geq 2} \chi(\Mbar_{0, n+1}) \frac{y^n}{n!}, \]
so \[g(y) \circ \left(y-\sum_{n\geq 2} \chi(\M_{0,n+1}) \frac{y^n}{n!}\right) = \left(y-\sum_{n\geq 2} \chi(\M_{0,n+1}) \frac{y^n}{n!}\right) \circ g(y) = y,\]
by \cite{GetzlerGenusZero}. By \cite[Theorem 4.1]{GetzlerSemiClassical} we have 
\begin{equation*}
    \sum_n \chi(\Mbar_{1,n}) \frac{y^n}{n!}=-\frac{1}{12}\log(1+g(y)) - \frac12\log(1-\log(1+g(y))) + \varepsilon(g(y)), 
\end{equation*}
so we derive
\begin{align*}
    \sum_n \chi(\Mbar_{1,0|n}) \frac{y^n}{n!}&=\left(\sum_n \chi(\Mbar_{1,n}) \frac{y^n}{n!}\right) \circ \left(y-\sum_{n\geq 2} \chi(\M_{0,n+1}) \frac{y^n}{n!}\right) \circ (e^y-1) \\
    &=\left(-\frac{1}{12}\log(1+y) - \frac12\log(1-\log(1+y)) + \varepsilon(y)\right) \circ (e^y-1) \\
    &=-\frac{y}{12} - \frac12\log(1-y) + \varepsilon \circ (e^y-1),
\end{align*}
as claimed.
\end{proof}
The following corollary indicates that eliminating rational tails reduces the topological complexity of the moduli space.
\begin{cor}
We have the asymptotic formulas
\[\chi(\Mbar_{1, 0|n}) \sim \frac{(n - 1)!}{2}, \]
and
\[
\frac{\chi(\Mbar_{1,0|n})}{\chi(\Mbar_{1,n})} \sim 2(e - 2)^n.
\]
In particular, \[\lim_{n \to \infty} \frac{\chi(\Mbar_{1,0|n})}{\chi(\Mbar_{1,n})} = 0. \]
\label{cor: chiComparison}
\end{cor}

\begin{proof}
If we think of $y$ as a complex variable, the function $f(y)-(- \log(1 - y)/2)$ is an entire function. By \cite[Theorem 2.4.3]{wilf}, the values $\chi(\Mbar_{1,0|n})/n!$ are approximated by the power series coefficients of the function $- \frac12\log(1-y)$ about the origin. Therefore, we have
\begin{equation*}
    \chi(\Mbar_{1,0|n}) \sim \frac{(n-1)!}{2}.
\end{equation*}
By \cite[Corollary 4.2]{GetzlerSemiClassical},
\begin{equation*}
    \chi(\Mbar_{1,n}) \sim \frac{(n-1)!}{4(e-2)^n}(1+Cn^{-1/2} + O(n^{-3/2})) \sim \frac{(n-1)!}{4(e-2)^n},
\end{equation*} where $C$ is a constant. Therefore
\begin{equation*}
    \frac{\chi(\Mbar_{1,0|n})}{\chi(\Mbar_{1,n})} \sim \frac{(n-1)!}{2}\frac{4(e-2)^n}{(n-1)!}= 2(e- 2)^n,
\end{equation*}
as we wanted to show.
\end{proof}

Corollary \ref{cor: chiComparison} implies that when $n$ is large, the ``all light points" moduli space $\Mbar_{1, 0|n}$ has much less cohomology than the Deligne--Mumford moduli space $\Mbar_{1, n}$. It is natural to ask whether the same holds for higher genus.
\begin{quest}\label{asymptoticsQuestion}
Find an asymptotic formula for the quotient
\[\frac{\chi(\Mbar_{g, 0|n})}{\chi(\Mbar_{g, n})} \]
for all $g \geq 2$. Do we have
\[ \lim_{n \to \infty} \frac{\chi(\Mbar_{g, 0|n})}{\chi(\Mbar_{g, n})} = 0 \]
for all such $g$?

\end{quest}

% The above proof indicates that the Euler characteristic of $\Mbar_{1, n}$ is much bigger than that of $\Mbar_{1, 0|n}$. Indeed, the proof shows that we have
% \begin{equation}\label{PreciseAsymptoticComparison}
%     \chi(\Mbar_{1, n}) \sim \frac{1}{2} \left(\frac{1}{e - 2}\right)^n \cdot \chi(\Mbar_{1, 0|n}) \approx \frac{(1.3922)^n}{2} \cdot \chi(\Mbar_{1, 0|n}).
% \end{equation}
\subsection{Tables of data}
We conclude the paper by including three tables containing sample calculations, done with SageMath, based on Theorem \ref{GeneratingFunctionFormulas}.\footnote{Sage code for these computations is available at \url{https://sites.google.com/view/siddarthkannan/research}} The first, Table \ref{EquivariantGenusOneTable}, contains the $(S_m \times S_n)$-equivariant Hodge polynomial of $\Mbar_{1, m|n}$ for $m + n \leq 5$. These rely on the calculation of the series $\overline{\mathsf{b}}_1$ by Getzler \cite{GetzlerSemiClassical}. For $n \leq 10$, the mixed Hodge structures on the cohomology groups of the moduli space $\Mbar_{1, n}$ are polynomials in $\mathsf{L} = H^{2}_c(\A^1;\C)$, the mixed Hodge structure of the affine line. A consequence is that $\Mbar_{1, n}$ has only even dimensional cohomology for $n \leq 10$, and only the diagonal Hodge numbers $\dim H^{p,p}$ are nonzero. By Theorem \ref{GeneratingFunctionFormulas}, the same is true for $\Mbar_{1, m|n}$ for $m + n \leq 10$. Therefore, Table \ref{EquivariantGenusOneTable} displays the equivariant Poincar{\'e} polynomial \[E^{S_m \times S_n}_{\Mbar_{g, m|n}}(t, t),\] and the Hodge polynomial can be recovered by setting $t^2 = uv$.

Table \ref{NumericalGenusOneTable} contains the non-equivariant Hodge polynomial of $\Mbar_{1, 0|n}$ for $n \leq 11$, computed with Corollary \ref{NumericalFormula} and Getzler's calculation of $\overline{b}_1$. %We highlight the case where all markings are light, because this is the case in which the moduli space has no curves with rational tails. It seems that this simplification greatly reduces the complexity of the moduli space: 
By Corollary \ref{cor: chiComparison} and (\ref{PreciseAsymptoticComparison}), one might expect $\Mbar_{1,0|n}$ to have less cohomology than $\Mbar_{1, n}$ (this is not a direct consequence; both spaces may have odd cohomology). Indeed, comparing with the table \cite[p.491]{GetzlerSemiClassical}, we observe that $\dim H^*(\Mbar_{1, 0|10}) = 232,076$ while $\dim H^*(\Mbar_{1, 10}) = 16,275,872$. One also notes that just as in the case of $\Mbar_{1, 11}$, the space $\Mbar_{1, 0|11}$ has odd-dimensional cohomology; this is true of $\Mbar_{1, m|n}$ whenever $m + n = 11$.

Finally, Table \ref{WeightZeroGenusTwoTable} contains the $(S_m \times S_n)$-equivariant compactly supported weight zero Euler characteristic of $\calM_{2, m|n}$ for $m + n \leq 6$, which is equal to
\[E^{S_m \times S_n}_{\calM_{2, m|n}}(0,0), \]
the constant term of the Hodge--Deligne polynomial. We also include the numerical weight zero Euler characteristic. This table was computed using the first part of Theorem \ref{GeneratingFunctionFormulas}, together with the formula of Chan et al. for $E_{\calM_{g, n}}^{S_n}(0,0)$ \cite{CFGP}. We also note that this table and our techniques apply to compute the equivariant Euler characteristic
\[\chi^{S_m \times S_n}(\Delta_{g, m|n}) := \sum_{i} (-1)^i \ch_{m|n}(H^i(\Delta_{g, m|n};\Q)) \in \Lambda^{(2)}, \]
where $\Delta_{g, m|n}$ is the tropical heavy/light Hassett space, studied in \cite{CHMRtropical, tropicalhassett1,tropicalhassett, KKLChow}. Indeed, one has
\[\chi^{S_m \times S_n}(\Delta_{g, m|n}) = s_m^{(1)}s_n^{(2)} - E^{S_m \times S_n}_{\calM_{g, m|n}}(0,0) \]
when $\Delta_{g, m|n}$ is connected, which holds when $g \geq 1$, and when $g = 0$ and $m + n > 4$. See \cite[\S 4]{tropicalhassett}. It is natural to ask whether there is a way to calculate $E^{S_m \times S_n}_{\calM_{g, m|n}}(0,0)$ which does not rely on the formula of \cite{CFGP}, together with plethysm.

\begin{quest}\label{weightzeroQuestion}
Can the graph-theoretic methods of \cite{CFGP} be adapted to give a closed formula for the weight zero compactly supported Euler characteristic $E^{S_m \times S_n}_{\calM_{g, m|n}}(0,0)$?
\end{quest}

We exclude the case $n = 1$ from Tables \ref{EquivariantGenusOneTable} and \ref{WeightZeroGenusTwoTable}, as
\[E^{S_{m} \times S_1}_{\Mbar_{g, m|1}}(u,v) = \left( \frac{\partial E^{S_{m + 1}}_{\Mbar_{g, m + 1}}(u,v) }{\partial p_1} \right)^{(1)} \cdot s_{1}^{(2)}, \]
and the analogous formula holds for the open moduli spaces.

% \section{Frobenius characteristic reference}
% Let $\symring = \varprojlim_n \Q[[x_1,\dots,x_n]]^{S_n}$ be the ring of formal symmetric power series. Let $V$ be an $S_{m+n}$-representation. We want to calculate
% \[\ch_{m,n}(\mathrm{Res}^{S_{m + n}}_{S_m \times S_n} V) \in \symring \otimes \symring. \]
% Let $p_i$ be the $i$th power sum in the first copy of $\symring$ and let $q_i$ for the $i$th power sum in the second copy. Then we claim that
% \[\ch_{m,n}(\mathrm{Res}^{S_{m + n}}_{S_m \times S_n} V) = \frac{1}{n!} \left(\sum_{\lambda \vdash n} \frac{\partial^{|\lambda|}}{\partial p_{\lambda}}(\ch_{m + n}(V)) q_{\lambda} \right), \]
% where if $\lambda = (1^{\lambda_1}2^{\lambda_2}\cdots)$ we put
% \[p_{\lambda} = \prod_{i = 1}^{n} p_i^{\lambda_i} \]
% and $|\lambda| = \sum \lambda_i$. We are supposing that $\ch_{m + n}(V)$ is given in terms of the power sums.

% \section{Calculations}
% The following calculations were performed with SageMath.
\begin{landscape}
% \usepackage{multirow}
\renewcommand{\arraystretch}{2}
\begin{table}[]
\begin{tabular}{|c|l|}
\hline
$\left(m,n\right)$                   & \multicolumn{1}{c|}{$E^{S_m \times S_n}_{\Mbar_{1, m|n}}\left(t, t\right)$}                                                                                                  \\ \hline
$\left(0,2\right)$                   & $\left(t^4+2t^2+1\right)s_2^{\left(2\right)}$                                                                                                                                \\ \hline
$\left(0, 3\right)$                  & $\left(t^4+t^2\right)s_{2, 1}^{\left(2\right)} + \left(t^6+2t^4+2t^2+1\right)s_3^{\left(2\right)}$                                                                                                 \\ \hline
$\left(1, 2\right)$                  & $\left(t^4+t^2\right)s_1^{\left(1\right)}s_{1,1}^{\left(2\right)} + \left(t^6+4t^4 + 4t^2+1\right)s_1^{\left(1\right)}s_2^{\left(2\right)}$                                                                              \\ \hline
$\left(0, 4\right)$                  & $\left(t^6+2t^4+t^2\right)s_{2, 2}^{\left(2\right)} + \left(t^6+2t^4+t^2\right)s_{3, 1}^{\left(2\right)} + \left(t^8+2t^6+3t^4+2t^2+1\right)s_4^{\left(2\right)}$                                                        \\ \hline
$\left(1, 3\right)$                  & $\left(3t^6+6t^4+3t^2\right)s^{\left(1\right)}_1 s^{\left(2\right)}_{2, 1} + \left(t^8+5t^6+9t^4+5t^2+1\right)s_1^{\left(1\right)} s_3^{\left(2\right)}$                                                                \\ \hline
$\left(2, 2\right)$                  & $\left(\left(t^6+2t^4+t^2\right)s_{1, 1}^{\left(1\right)} + \left(2t^6+4t^4+2t^2\right)s^{\left(1\right)}_{2}\right)s^{\left(2\right)}_{1,1} + \left(\left(2t^6+4t^4+2t^2\right)s^{\left(1\right)}_{1, 1} + \left(t^8+7t^6+13t^4+7t^2+1\right)s^{\left(1\right)}_{2}\right)s^{\left(2\right)}_{2}$                                                          \\ \hline
$\left(0, 5\right)$                  & $\left(t^6+t^4\right)s_{2,2,1}^{\left(2\right)} + \left(t^8+3t^6+3t^4+t^2\right)s_{3, 2}^{\left(2\right)} + \left(t^8+3t^6+3t^4+t^2\right)s_{4, 1}^{\left(2\right)} + \left(t^{10}+2t^8+3t^6+3t^4+2t^2+1\right)s_5^{\left(2\right)}$            \\ \hline
\multirow{2}{*}{$\left(1, 4\right)$}                 & $\left(2t^6+2t^4\right)s_1^{\left(1\right)}s_{2, 1, 1}^{\left(2\right)} + \left(2t^8+8t^6+8t^4+2t^2\right)s_1^{\left(1\right)}s_{2, 2}^{\left(2\right)}$                                                                 \\
                          & \phantom{this is space}$+\left(4t^8+14t^6+14t^4+4t^2\right)s_1^{\left(1\right)}s_{3,1}^{\left(2\right)} + \left(t^{10}+6t^8+15t^6+15t^4+6t^2+1\right)s_1^{\left(1\right)}s_4^{\left(2\right)}$                                                    \\ \hline
\multirow{2}{*}{$\left(2, 3\right)$}                 & $\left(t^6+t^4\right)s^{\left(1\right)}_{1, 1}s^{\left(2\right)}_{1, 1, 1} + \left(t^6+t^4\right)s^{\left(1\right)}_{2}s^{\left(2\right)}_{1, 1, 1} + \left(\left(2t^8+9t^6+9t^4+2t^2\right)s^{\left(1\right)}_{1, 1}+\left(5t^8+20t^6+20t^4+5t^2\right)s^{\left(1\right)}_{2}\right)s^{\left(2\right)}_{2, 1}$                                              \\
\multicolumn{1}{|l|}{}    & \phantom{this is space}$ +\left(\left(3t^8+11t^6+11t^4+3t^2\right)s^{\left(1\right)}_{1, 1} + \left(t^{10}+9t^8+26t^6+26t^4+9t^2+1\right)s^{\left(1\right)}_{2}\right)s^{\left(2\right)}_{3}$ \\ \hline
\multirow{2}{*}{$\left(3, 2\right)$}                   & $t^{10} s_3^{\left(1\right)}s_2^{\left(2\right)} + t^8\left(2s_{2,1}^{\left(1\right)} + 3s_3^{\left(1\right)}\right)s_{1,1}^{\left(2\right)} + \left(5s_{2,1}^{\left(1\right)} + 10s_{3}^{\left(1\right)}\right)s_{2}^{\left(2\right)}$                          \\
\multicolumn{1}{|l|}{}    & \phantom{this is space} $+t^6\left(\left(s_{1, 1, 1}^{\left(1\right)} + 9s_{2, 1}^{\left(1\right)} + 12s_{3}^{\left(1\right)}\right)s_{1, 1}^{\left(2\right)} + \left(s_{1,1,1}^{\left(1\right)} + 20s_{2, 1}^{\left(1\right)} + 30s_{3}^{\left(1\right)}\right)s_{2}^{\left(2\right)}\right) + \cdots$ \\ \hline
\end{tabular}
\caption{The $(S_m \times S_n)$-equivariant Poincare polynomials of $\Mbar_{1, m|n}$ for $m + n \leq 5$. The omitted terms are determined by Poincar{\'e} duality.}
\label{EquivariantGenusOneTable}
\end{table}
\end{landscape}

\begin{landscape}
\renewcommand{\arraystretch}{2}
\begin{table}[]
\begin{tabular}{|c|l|}
\hline
$n$  & \multicolumn{1}{c|}{$E_{\Mbar_{1, 0|n}}(u, v)$}                                                                                                                            \\ \hline
$1$  & $uv + 1$                                                                                                                                                                   \\ \hline
$2$  & $u^2v^2 + 2uv + 1$                                                                                                                                                         \\ \hline
$3$  & $u^3v^3 + 4u^2v^2 + 4uv + 1$                                                                                                                                               \\ \hline
$4$  & $u^4v^4 + 7u^3v^3 + 13u^2v^2 + 7uv + 1$                                                                                                                             \\ \hline
$5$  & $u^5v^5 + 11u^4v^4 + 35u^3v^3 + 35u^2v^2 + 11uv + 1$                                                                                                                       \\ \hline
$6$  & $u^6v^6 + 16u^5v^5 + 81u^4v^4 + 140u^3v^3 + 81u^2v^2 + 16uv + 1$                                                                                                           \\ \hline
$7$  & $u^7v^7 + 22u^6v^6 + 168u^5v^5 + 476u^4v^4 + 476u^3v^3 + 168u^2v^2 + 22uv + 1$                                                                                             \\ \hline
$8$  & $u^8v^8 + 29u^7v^7 + 323u^6v^6 + 1456u^5v^5 + 2458u^4v^4 + 1456u^3v^3 + 323u^2v^2 + 29uv + 1$                                                                              \\ \hline
$9$  & $u^9v^9 + 37u^8v^8 + 591u^7v^7 + 4201u^6v^6 + 11901u^5v^5 + 11901u^4v^4 + 4201u^3v^3 + 591u^2v^2 + 37uv + 1$                                                               \\ \hline
$10$ & $u^{10}v^{10} + 46u^9v^9 + 1051u^8v^8 + 11850u^7v^7 + 55975u^6v^6 + 94230u^5v^5 + 55975u^4v^4 + 11850u^3v^3 + 1051u^2v^2 + 46uv + 1$                                      \\ \hline
$11$ & $u^{11}v^{11} + 56u^{10}v^{10} + 1848u^9v^9 + 33451u^8v^8 + 258940u^7v^7 + 710512u^6v^6 - u^{11} - v^{11} + 710512u^5v^5 + 258940u^4v^4 + \cdots$ \\ \hline
\end{tabular}
\caption{The Hodge polynomial of $\Mbar_{1, 0|n}$ for $n \leq 11$. }
\label{NumericalGenusOneTable}
\end{table}
\end{landscape}

\begin{landscape}
\renewcommand{\arraystretch}{2}
\begin{table}[]
\begin{tabular}{|c|l|c|}
\hline
$\left(m, n\right)$ & \multicolumn{1}{c|}{$E_{\calM_{2, m|n}}^{S_m \times S_n}\left(0,0\right)= \sum_{k = 0}^{6 + 2\left( m + n\right)} \left(-1\right)^k\mathrm{ch}_{m,n}\left(W_0H^k_c\left(\M_{2, m|n};\C \right)\right)$} &$E_{\calM_{2, m|n}}(0,0)$ \\ \hline
% $\left(0,1\right)$  & $0$ & $0$\\ \hline
$\left(0,2\right)$  & $-s_{2}^{\left(2\right)}$ &$-1$ \\ \hline
$\left(0,3\right)$  & $-s_{2, 1}^{\left(2\right)} - s_3^{\left(2\right)}$&$-3$\\ \hline
$\left(1,2\right)$  & $-s_{1}^{\left(1\right)}s_{2}^{\left(2\right)}$ &$-1$\\ \hline
$\left(0,4\right)$  & $-s_{2, 1, 1}^{\left(2\right)} - s_{2, 2}^{\left(2\right)} - s_{3, 1}^{\left(2\right)} - s_{4}^{\left(2\right)}$&$-9$\\ \hline
$\left(1,3\right)$  & $-s_{1}^{\left(1\right)}\left(s_{1,1,1}^{\left(2\right)} + s_{2,1}^{\left(2\right)}\right)$ &$-3$\\ \hline
$\left(2,2\right)$  & $\left(-s_{1, 1}^{\left(1\right)} - s_{2}^{\left(1\right)}\right)s_{1,1}^{\left(2\right)} + \left(-s_{1, 1}^{\left(1\right)} + s_{2}^{\left(1\right)}\right)s_{2}^{\left(2\right)}$ & $-2$\\ \hline
$\left(0,5\right)$  & $-s_{2, 1, 1, 1}^{\left(2\right)} - s_{2, 2, 1}^{\left(2\right)} - s_{3, 1, 1}^{\left(2\right)} - s_{3, 2}^{\left(2\right)} - s_{4, 1}^{\left(2\right)} - s_{5}^{\left(2\right)}$ & $-25$\\ \hline
$\left(1,4\right)$  & $-s_{1}^{\left(1\right)}s_{2, 1, 1}^{\left(2\right)}$ &$-3$\\ \hline
$\left(2,3\right)$  & $s_{2}^{\left(1\right)}\left(s_{1,1,1}^{\left(2\right)} + s_{3}^{\left(2\right)}\right) + \left(2s_{2}^{\left(1\right)} - s_{1,1}^{\left(1\right)}\right)s_{2, 1}^{\left(2\right)}$&$4$\\ \hline
$\left(3,2\right)$  & $\left(s_{2, 1}^{\left(1\right)} + 2s_{3}^{\left(1\right)}\right)s_{1, 1}^{\left(2\right)} + \left(s_{2, 1}^{\left(1\right)} + 2s_{3}^{\left(1\right)}\right)s_{2}^{\left(2\right)}$&$8$\\ \hline
$\left(0,6\right)$  & $-s_{1, 1, 1, 1, 1, 1}^{\left(2\right)} - s_{2, 1, 1, 1, 1}^{\left(2\right)} - s_{2, 2, 1, 1}^{\left(2\right)} - s_{2, 2, 2}^{\left(2\right)} - s_{3, 1, 1, 1}^{\left(2\right)} - s_{3, 2, 1}^{\left(2\right)} - s_{4, 1, 1}^{\left(2\right)} - s_{4, 2}^{\left(2\right)} - s_{5, 1}^{\left(2\right)} - s_{6}^{\left(2\right)}$ &$-71$\\ \hline
$\left(1,5\right)$  & $-s_1^{\left(1\right)}\left(2s_{1,1,1,1,1}^{\left(2\right)} + s_{2,1,1,1}^{\left(2\right)} + s_{2,2,1}^{\left(2\right)}\right)$ &$-11$ \\ \hline
$\left(2,4\right)$  & $\left(-3s_{1, 1}^{\left(1\right)} - s_{2}^{\left(1\right)}\right)s_{1, 1, 1, 1}^{\left(2\right)} + \left(-4s_{1, 1}^{\left(1\right)} + 2s_2^{\left(1\right)}\right)s_{2,1,1}^{\left(2\right)} + \left(-s_{1, 1}^{\left(1\right)} - s_{2}^{\left(1\right)}\right)s_{2, 2}^{\left(2\right)} + \left(-s_{1, 1}^{\left(1\right)} + s_{2}^{\left(1\right)}\right)s_{3, 1}^{\left(2\right)}$ & $-14$\\ \hline
$\left(3,3\right)$  & $\left(-5s_{1, 1, 1}^{\left(1\right)} - 3s_{2, 1}^{\left(1\right)} + s_3^{\left(1\right)}\right)s_{1,1,1}^{\left(2\right)} + \left(-2s_{1, 1, 1}^{\left(1\right)} - 4s_{2, 1}^{\left(1\right)}\right)s_{2, 1}^{\left(2\right)} + \left(s_{1, 1, 1}^{\left(1\right)} - s_{2, 1}^{\left(1\right)} - s_{3}^{\left(1\right)}\right)s_3^{\left(2\right)}$ &$-32$ \\ \hline
$\left(4,2\right)$  & $\left(-3s_{1, 1, 1, 1}^{\left(1\right)} - 5s_{2, 1, 1}^{\left(1\right)} - 3s_{2, 2}^{\left(1\right)} - 2s_{3, 1}^{\left(1\right)}\right)s_{1,1}^{\left(2\right)} + \left(-s_{1, 1, 1, 1}^{\left(1\right)} - 5s_{2, 2}^{\left(1\right)} - 2s_{3, 1}^{\left(1\right)} - 3s_4^{\left(1\right)}\right)s_{2}^{\left(2\right)}$ & $-50$\\ \hline
\end{tabular}
\caption{The $(S_m \times S_n)$-equivariant and numerical weight zero compactly supported Euler characteristics of $\calM_{2, m|n}$.}
\label{WeightZeroGenusTwoTable}
\end{table}
\end{landscape}

\bibliographystyle{amsalpha}
\bibliography{reference}

\end{document}